\documentclass[hidelinks,onefignum,onetabnum]{siamart250211}
\usepackage{amsfonts}
\usepackage{bm}
\usepackage{url}

\newsiamremark{remark}{Remark}
\DeclareMathOperator{\diag}{diag}

\newcommand{\figref}[1]{\figurename~\hyperref[#1]{\ref{#1}}}

\newcommand{\EPS}{\varepsilon}
\newcommand{\PD}[2]{\frac{\partial{#1}}{\partial{#2}}}
\newcommand{\conj}[1]{\overline{#1}}

\newcommand{\lp}{\left(}
\newcommand{\rp}{\right)}
\newcommand\bx{\boldsymbol{x}}

\newcommand{\Smat}{\cI}
\newcommand{\Strunc}{\Smat^{\textrm{trunc}}}
\newcommand{\Pout}{\mathcal P}

\newcommand{\uin}{u^{\textrm{in}}}
\newcommand{\cD}{\mathcal D}
\newcommand{\cS}{\mathcal S}
\newcommand{\cF}{\mathcal F}
\newcommand{\cP}{\mathcal P}
\newcommand{\cI}{\mathcal I}
\newcommand{\restr}[2]{\left.#1\right|_{#2}}
\newcommand\bbR{\mathbb R}

\newcommand{\bs}{\boldsymbol}

\def\ccm{Center for Computational Mathematics, Flatiron Institute, Simons Foundation,
  New York, New York 10010}
  
\headers{Computing the scattering matrix of waveguide circuits}{T. Goodwill, S. Jiang, M. Rachh, K. Sugita}

\title{A domain decomposition method for computing the scattering matrix of waveguide circuits}

\author{Tristan Goodwill\thanks{Department of Statistics and CCAM, University of Chicago, Chicago, NY 
  (\email{tgoodwill@uchicago.edu}).}
\and Shidong Jiang\thanks{\ccm\, 
(\email{sjiang@flatironinstitute.org}).}
\and Manas Rachh\thanks{Department of Mathematics, Indian Institute of Technology Bombay, Mumbai, India 
(\email{mrachh@iitb.ac.in}).} 
\and Kosuke Sugita\thanks{\ccm\, 
(\email{ksugita@flatironinstitute.org}).} }

\begin{document}

\maketitle

\begin{abstract}
We analyze and develop numerical methods for time-harmonic wave scattering
in metallic waveguide structures of infinite extent. We show that
radiation boundary conditions formulated via projectors onto outgoing
modes determine the coefficients of propagating modes uniquely, even
when the structure supports trapped modes. Building on this, we introduce
a fast divide-and-conquer solver that constructs solution operators
on subdomains as impedance-to-impedance maps and couples them by enforcing
continuity conditions across their interfaces. For Dirichlet waveguides,
the computation of impedance-to-impedance maps requires the solution
of mixed Dirichlet–Impedance boundary value problems. We construct
a second-kind Fredholm integral equation that avoids near-hypersingular
operators, requiring only integral operators whose kernels are at most
weakly singular. Numerical experiments on large structures with many
circuit elements demonstrate substantial efficiency gains: the proposed
approach typically outperforms state-of-the-art fast iterative and
fast direct solvers by one to two orders of magnitude.
\end{abstract}


\section{Introduction}
Accurate simulations for integrated photonic circuits play a vital part
in their design process. For this reason, there exists a large variety
of methods for simulating these systems, such as finite difference
time domain methods~\cite{JCP-1986-Berenger,MathComp-1986-Higdon,IEEETransEC-1981-Mur,IEEETransAntenna-1966-Yee},
finite element methods~\cite{Book-Brenner-Scott,Book-Itoh,Book-Koshiba,Book-Yamashita,Yeh1979SinglemodeOW,Young-1988,Book-Zienkiewicz},
beam propagation methods~\cite{ComCompP-2009-Antonie,JACT-2010-Antonie,RMMAN-2017-Bamberger,JCP-2017-Fan,Kragl-1992,Kumbhakar-2008,JOSAA-2009-Yioultsis},
and boundary integral equation methods~\cite{IEEE-Boriskina-2002, sideris2019ultrafast,garza2020boundary,tsukerman2008computational,yasumoto2018electromagnetic}.
A description of most of these methods can be found in~\cite{Book-Okamoto}.

In the time-harmonic setting, the electromagnetic waves in these devices
satisfy Maxwell's equations with appropriate boundary conditions. In
two dimensions, and in the transverse electric or transverse magnetic
modes, the waves satisfy Helmholtz boundary value problems. A key
difficulty in the simulation of these devices is their large size
as measured in wavelengths of propagating light --- typically the wavelength
is about a micrometer, while the devices are millimeters in size.
This corresponds to a computational domain that is about $1000$ wavelengths
across. As this is a high frequency problem, the complexity of many
of the above methods tends to scale at least quadratically in the
size of the computational domain.

A device typically consists of several photonic circuit elements and
a collection of ports which support a finite number of propagating
modes. In the ports, there also is an evanescent field, which can
be neglected as long as the length of the ports is $O(1)$ wavelengths
in size. Moreover, the design of these devices tends to be modular,
i.e., complicated photonic devices are assembled from a few identical
circuit elements and rectangular connectors (see~\figref{fig:domain_def}). 

In this paper, we develop an efficient method for the simulation of
such metallic waveguide systems in two dimensions. The solver consists
of two steps: dividing the domain into its individual circuit elements
and storing their solution operators as impedance-to-impedance maps,
followed by ``gluing'' the solutions across circuit elements by enforcing
continuity of the potential and its normal derivative. The impedance-to-impedance
maps of the circuit elements are projected onto an appropriate bases
whose cardinality depends only on the number of propagating modes
supported in the element and its neighbors (i.e., other circuit elements
that share an edge with it), and the size of the separator. More
importantly, the cardinality of the basis for representing the solution
in the individual components is independent of the size of element
as measured in wavelengths. This significantly reduces the size of
linear system to be solved in the second stage --- for many practical
devices, dense linear algebra methods can be used for its solution.
A key benefit of this modularization is that instead of solving a
partial differential equation (PDE) on a domain with~$N$ photonic
circuit elements, we can solve~$N$ uncoupled PDEs, one for each
photonic element, followed by the inversion of a significantly smaller
linear system.

{\bf Related work}:
The solution operators on the individual circuit elements are also
often referred to as a scattering matrix which map ``incoming'' data
to ``outgoing'' data. They have proven to be an extremely useful
mathematical tool and concept for studying any system that exhibits
the linear relationship between its input and output data. There is
a long history of research on the scattering theory in mathematical
physics, we do not seek to review the literature extensively, but
some recent examples include~\cite{SIAMAppMath-2001-Bendhia,SIAMAppMath-2011-Bendhia,SIAMCompAppMath-2007-Chandler,RSLPS-1998-Chandler,gimbutas2013fast,bremer2015high,borges2025construction}.
A discussion of their uses in the design and development of microwave
network systems can be found in~\cite{Book-Rao}, for example.

The solver presented in this work is also related to the Hierarchical
Poincar\'{e} Steklov (HPS) solvers for computing solutions of time-harmonic
wave-scattering problems using piecewise high-order spectral elements,
see~\cite{gillman2014direct,gillman2015spectrally,martinsson2015hierarchical,fortunato2021ultraspherical,fortunato2024high},
for example. Our method varies in two key regards. First, the discretization
of the individual circuit elements is based on an integral equation
formulation rather than a direct discretization of the PDE, and second
a hierarchical solve is not necessary in this context owing to the
small system size after constructing the solution operators of the
circuit elements. Akin to the method presented here, the HPS method
also has a gluing stage where the solution operators computed on sub-domains
are combined by enforcing smoothness of solutions across their boundaries.
This approach can easily be adapted for the solution of very large
networks of circuit elements when direct methods become computationally
expensive.

Finally, a related approach has been applied for the simulation Stokes
flow in complex branched structures~\cite{wang2025scattering}. In
this case, similar to metallic waveguides, there is one static mode
in each ``circuit'' element, while the rest of the modes decay exponentially.
The work assumes that the circuit elements are separated by a sufficiently
long straight channel so that all evanescent modes decay below a prescribed
tolerance --- we make no such assumption.

{\bf Contributions}:
There are four main contributions of this work. First, we present
a well-posed PDE formulation of the problem, including appropriate
radiation conditions at infinity. Following existing work, the outgoing
radiation conditions at infinity are expressed in terms of projectors
onto the propagating modes on the boundaries of the ports. The novelty
here is a proof of uniqueness of the coefficients of outgoing propagating
modes for any incoming data. While trapped modes are known to exist
in certain geometries, they do not project onto the propagating part
of the solution.  

Second, the construction of the impedance-to-impedance maps on the
circuit elements requires the solution of a mixed Dirichlet-Impedance
boundary value problem. Standard integral equations tend to require
the evaluation of near-hypersingular integrals. Inspired by the work
of~\cite{greengard2014fast}, we introduce a fictitious symmetric
boundary along the Dirichlet segment near the mixed-boundary junctions
to cancel the nearly hypersingular contributions on the adjacent impedance
segment. This modification results in an integral operator which is
Fredholm and only requires evaluation of integral operators with weakly
singular kernels.

Third, by retaining an appropriate number of decaying (evanescent)
modes in the impedance-to-impedance map, the connectors between circuit
elements can be made arbitrarily short without loss of accuracy. We
provide a heuristic, backed by numerical experiments, to empirically
determine the number of terms required in the basis. 

Finally, our divide-and-conquer domain decomposition approach is
highly efficient. We demonstrate its efficiency by comparing it to
three other methods in terms of the CPU-time required for computing
solutions on electromagnetically large photonic devices --- (i) a
fast iterative solver accelerated by wide-band fast multipole methods
(FMMs)~\cite{crutchfield2006remarks,fmm2d}, (ii) a fast direct
solver based on recursive skeletonization~\cite{SISC-Ho-2012,ho2020flam},
and (iii) a hybrid solver that uses a low-accuracy fast-direct solver
as a preconditioner for the FMM-accelerated iterative solver. Our
approach outperforms these solvers by one to two orders of magnitude.

The paper is organized as follows. Section~\ref{sec:setup} formulates
the problem and proves that the scattering matrix is well-defined.
Section~\ref{sec:I2Imap} introduces the impedance-to-impedance
map and a domain-decomposition strategy for its computation. Section~\ref{sec:bie}
constructs a Fredholm integral equation of the second kind for the
mixed Dirichlet–Impedance problem. Section~\ref{sec:alg} outlines
the complete numerical algorithm and analyzes its complexity. Section~\ref{sec:num_exam}
presents numerical results and compares our method with state-of-the-art
fast solvers.

\section{Problem setup}\label{sec:setup}
In two dimensions, metallic waveguides can be modeled by the Helmholtz equation
with homogeneous Dirichlet or Neumann boundary conditions.
Let the interior of the waveguide circuit be denoted~$\Omega$, and let $\partial \Omega$ denote 
its boundary. Suppose that this domain can be split into a compact region~$\Omega_0$ 
(with boundary $\partial \Omega_{0}$) and a collection of semi-infinite rectangles~$\Omega_1,\ldots,\Omega_P$, 
which represent the input/output channels of the device (see~\figref{fig:domain_def}).
Suppose further that the channels continue into~$\Omega_0$ at least a
distance~$L$. Finally, let~$\gamma_p$ be the boundary between~$\Omega_0$ and~$\Omega_p$ 
and let~$\Gamma =\partial\Omega_0 \setminus\lp \cup_{p=1}^P \gamma_p\rp$.

\begin{figure}
    \centering
    \includegraphics[width=1\linewidth]{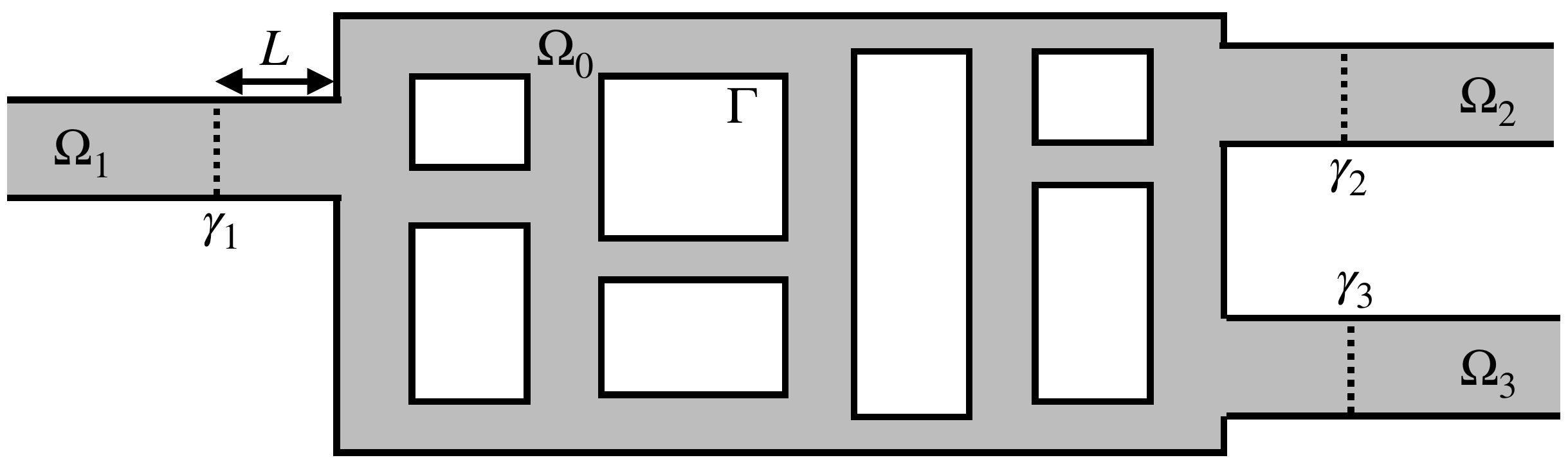}
    \caption{An example waveguide circuit. The compact region~$\Omega_0$ and three straight 
    channels~$\Omega_1,\Omega_2,$ and~$\Omega_3$. The regions~$\Omega_0$ and~$\Omega_j$ are separated 
    by a line segment~$\gamma_j$. The figure also illustrates the definition of the truncation length~$L$.}
    \label{fig:domain_def}
\end{figure}

For concreteness, consider the Dirichlet boundary value problem. Inside
$\Omega$, there is a potential $u$ that satisfies 
\begin{equation}
\begin{aligned}
\Delta u(\boldsymbol{x}) + k^2 u(\boldsymbol{x})&=0, \qquad \boldsymbol{x}\in \Omega, \\
u(\boldsymbol{x}) &= 0, \qquad \boldsymbol{x} \in \partial \Omega,
\end{aligned}\label{eq:dir_BVP}
\end{equation}
with appropriate conditions at infinity. The formulation of the conditions at
infinity requires understanding the nature of allowable solutions in each port
$\Omega_{p}$. 

The solutions in such rectangular ports are easily computed via separation 
of variables. Let~$(x_{p}, y_{p})$ denote local coordinates for the port~$\Omega_{p}$, 
obtained by an appropriate translation and rotation, such that 
\begin{equation}
\Omega_{p} = \bigg\{ (x_{p},y_{p})\in \mathbb{R}^{2}: \, x_{p} \in (0,\infty)\,, \text{ and }  y_{p} \in \left[-\frac{d_{p}}{2}, \frac{d_{p}}{2} \right] \bigg\} \, ,
\end{equation}
where~$d_p$ is the width of port~$p$, and $x_p=0$ corresponds to the boundary~$\gamma_p$.

Any solution $u \in H^1_{\mathrm{loc}}(\Omega_p)$ of the Helmholtz
equation in $\Omega_p$ with Dirichlet boundary conditions can be
written as an $H^1_{\mathrm{loc}}$-convergent series
\begin{equation}
    u(x_p,y_p) = \sum_{m=1}^\infty \lp c_{m,+}^{p}e^{i\beta_m^p x_p}+c_{m,-}^pe^{-i\beta_m^p x_p}\rp b_m^p(y_p), \label{eq:u_sep}
\end{equation}
where
\begin{equation}
    b_m^p(y_p) = \sqrt{\frac{2}{d_p}} \sin\lp\frac{m\pi}{d_p} \lp y_p+\frac {d_p}2\rp\rp \quad\text{and}\quad \beta_m^p = \sqrt{k^2-\lp\frac{m\pi}{d_p}\rp^2}.\label{eq:dir_modes}
\end{equation}
The basis functions~$e^{\pm i\beta_m^px_p}b_m^p(y_p)$ are referred to as the modes supported by~$\Omega_p$. 
For~$m$ small,~$\beta_m^p$ is real, and so the mode will propagate in the channel with constant modulus. 
We adopt the following convention: the coefficient~$c_{m,-}^p$ corresponds to a mode incident on~$\gamma_p$ 
(an “incoming” mode), whereas $c_{m,+}^p$ corresponds to a mode scattered away from~$\gamma_p$ 
(an “outgoing” propagating mode). For sufficiently large $m$,~$\beta_m^p$ is purely imaginary; adopting the 
standard square-root branch (so ${\rm Im}\beta_m^p>0$), the mode $e^{i\beta_{m}^{p}x_{p}}$ is evanescent 
and decays exponentially, whereas the mode $e^{-i\beta_{m}^{p}x_{p}}$ grows exponentially.  In what follows, 
let~$M_p$ denote the number of propagating modes for each subdomain $\Omega_p$. We will always assume 
that~$M_p>0$, i.e., $k d_p> \pi$.

Physically meaningful scattering problems in such domains correspond to the computation of an outgoing 
scattered field in response to incoming propagating modes. An outgoing scattered field is one which includes 
only the outgoing propagating modes, and the evanescent modes. More precisely, for each port, we introduce 
the following projection operators
\begin{equation}
    \cP_m^p u = \int_{\gamma_p} \bar{b}_m^p u\quad \text{and}\quad \cP_m^{'p} u = \int_{\gamma_p} \bar{b}_m^p \partial_{x_p} u\,.
\end{equation}
Using these, the coefficients in~\eqref{eq:u_sep} are given by
\begin{equation}
    c^p_{m,\pm} = \frac{1}{ 2i\beta_m^p}\lp \pm\cP_m^{'p} u +i\beta_m^p \cP_m^p u \rp .\label{eq:proj_coef}
\end{equation}
Suppose further that the field can be written as
\begin{equation}
    \begin{cases}
        u|_{\Omega_p} = \uin_p + u_p & p=1,\ldots, P\\
        u|_{\Omega_0} = u_0
    \end{cases}, \label{eq:u_piece}
\end{equation}
where~$\uin_p$ are prescribed incoming fields of the form
\begin{equation}
    \uin_p(x_p,y_p) = \sum_{m=1}^{M_p} c_{m,-}^p e^{-i\beta_m^p x_p} b_m^p(y_p)
\end{equation}
and the~$u_p$'s are outgoing in the sense that
\begin{equation}
    \lp\cP_m^{'p}  - i\beta_m^p \cP_m^p\rp u_p=0 \label{eq:piece_out}
\end{equation}
for all~$m$. By the relation~\eqref{eq:proj_coef}, this outgoing condition is equivalent to the statement 
that each~$u_p$ can be written as
\begin{equation}
    u_p(x_p,y_p) = \sum_{m=1}^{\infty} c_{m,+}^p e^{i\beta_m^p x_p} b_m^p(y_p).\label{eq:piece_outexp}
\end{equation}

For~$u$ defined by~\eqref{eq:u_piece} to be a solution of~\eqref{eq:dir_BVP}, $u_0,\ldots,u_P$ satisfy 
the following system of equations
\begin{equation}
    \begin{split}
        \Delta u_0 + k^2u_0=0 &\; \text{in }\Omega_0\\
        \Delta u_p + k^2u_p=0 & \;\text{in }\Omega_p\label{eq:piece_helm}
    \end{split}
\end{equation}
for~$p=1,\ldots,P$, along with the boundary conditions, 
\begin{equation}
        \begin{split}
                u_0 =0 & \;\text{on } \Gamma\\
        u_p=0 & \;\text{on } \partial\Omega_p \setminus \gamma_{p} \\
        u_0 = u_p+\uin_p & \;\text{on }\gamma_p\\
        \partial_{n_p}u_0 =\partial_{n_p} u_p+\partial_{n_p}\uin_p & \;\text{on }\gamma_p \,.
    \end{split} \label{eq:piece_jump}
\end{equation}
Furthermore $u_{p}$, $p=1,2,\ldots P$ are outgoing solutions in $\Omega_{p}$ and also 
satisfy~\eqref{eq:piece_out} on $\gamma_{p}$.

It is well known that there may be nontrivial solutions of the whole system~\eqref{eq:piece_out}, 
\eqref{eq:piece_helm}, and \eqref{eq:piece_jump} even with~$\uin_p = 0$ for all~$p$. These nontrivial 
solutions are called trapped modes and are known to be evanescent. We will therefore generally only 
be able to uniquely determine the coefficients of the outgoing propagating modes.
To prove this uniqueness, we introduce the following lemma, which is an immediate consequence of Green's identity.
\begin{lemma}[Generalized optical Theorem (See Lemma 3.2.1 in~\cite{Book-Nedelec})]\label{lem:opt_thm}
    If~$u_0$ satisfies
    \begin{equation}
            \Delta u_0+k^2 u_0 =0\;\text{in }\Omega_0
    \end{equation}
    then the following equation holds:
  \begin{equation}
    {\rm Im} \int_{\partial\Omega_0} u_0 \overline{\frac{\partial u_0}{\partial {n}}} ds = 0
    \label{eq:rellich-identity}
  \end{equation}
  where~${n}$ is the outward normal to~$\partial\Omega_0$.
\end{lemma}
This lemma allows us to prove the following uniqueness theorem.
\begin{theorem}\label{thm:uniq_prop}
    If~$\uin_p=0$ for all~$p=1,2,\ldots P$ and~$u_0,\ldots,u_P$ solve~\eqref{eq:piece_out}, 
    \eqref{eq:piece_helm}, and~\eqref{eq:piece_jump}, then 
    \begin{equation}
        \lp\cP_m^{'p}  + i\beta_m^p \cP_m^p\rp u_p=0
    \end{equation}
    for~$m=1,\ldots,M_p$ and all~$p$.
\end{theorem}
\begin{proof}
    Since~$\uin_p=0$ for all~$p$ and~$u_0|_\Gamma=0$, it follows from Lemma~\ref{lem:opt_thm} that
    \begin{equation}
        {\rm Im}\sum_p \int_{\gamma_p} u_p \conj{\frac{\partial u_p}{\partial x_p}} =  {\rm Im} \sum_p\int_{\gamma_p} u_0 \conj{\frac{\partial u_0}{\partial x_p}}=0.
        \label{eq:uupint}
    \end{equation}

    Substituting the outgoing expansion~\eqref{eq:piece_outexp} of $u_p$ into \eqref{eq:uupint} for each $p$ yields
    \begin{equation}
        {\rm Im} \sum_{p=1}^P \sum_{m,n=1}^\infty c_{m,+}^p  (-i\conj{\beta_n^p}) \conj{c_{n,+}^p} \int_{\gamma_p}b^p_m(y_p)b^p_n(y_p)dy_p =0.
    \end{equation}
  Using the orthonormality of the~$b_m^p$'s, we obtain
    \begin{equation}
        0= \sum_{p=1}^P \sum_{m=1}^\infty |c_{m,+}^p|^2 {\rm Im} \lp -i\conj{\beta_m^p} \rp.\label{eq:opt_prop}
    \end{equation}
    When~$m\leq M_p$, we have~$\beta_m^{p}\in\bbR$, so 
    \begin{equation}
         {\rm Im} \lp -i\conj{\beta_m^p} \rp  = -\beta_m^p.
    \end{equation}
    When~$m> M_{p}$, we have~$\beta_m^{p}\in i\bbR^+$, so 
        \begin{equation}
         {\rm Im} \lp -i\conj{\beta_m^p} \rp =   0.
    \end{equation}
    Plugging these expressions into~\eqref{eq:opt_prop} gives that
    \begin{equation}
        0 =  \sum_{p=1}^P \sum_{m=1}^{M_p} |c_{m,+}^p|^2 \beta_m^p.
    \end{equation}
    Because~$\beta_m^p>0$ for~$m\leq M_p$, we concluse that~$c_{m,+}^p=0$ for each~$m\leq M_p$. 
    This completes the proof.
\end{proof}
We now define the scattering matrix that completely characterizes the whole waveguide structure.
The global vector of \textbf{outgoing} propagating mode coefficients,
$\vec{c}_{+}$, is the concatenation of the coefficient vectors from
each of the $P$ ports: $\vec{c}_{+} = [\vec{c}_{+}^{\,1}; \vec{c}_{+}^{\,2}; \ldots; \vec{c}_{+}^{\,P}]$.
Each block $\vec{c}_{+}^{\,p}$ is a column vector containing the
$M_p$ outgoing mode coefficients for port $p$: $\vec{c}_{+}^{\,p} = [c^{p}_{1,+}; c^{p}_{2,+}; \ldots; c^{p}_{M_{p},+}]$.
Similarly, the global vector of \textbf{incoming} propagating mode
coefficients, $\vec{c}_{-}$, is structured in the same manner: $\vec{c}_{-} = [\vec{c}_{-}^{\,1}; \vec{c}_{-}^{\,2}; \ldots; \vec{c}_{-}^{\,P}]$.
Each corresponding block $\vec{c}_{-}^{\,p}$ contains the $M_p$
incoming mode coefficients for port $p$: $\vec{c}_{-}^{\,p} = [c^{p}_{1,-}; c^{p}_{2,-}; \ldots; c^{p}_{M_{p},-}]$.

\begin{definition} 
The scattering matrix, denoted by $\cS$, is the linear operator that maps the incoming coefficients 
vector~$\vec{c}_{-}$ to the outgoing coefficients vector~$\vec{c}_{+}$, i.e.,
\begin{equation}
\vec{c}_{+} = \cS \vec{c}_{-}.
\end{equation}
\end{definition}
Theorem~\ref{thm:uniq_prop} shows that $\cS$ is well-defined.
\subsection{Neumann boundary conditions}\label{sec:neumann_BC}
The above analysis extends to photonic circuits with Neumann boundary conditions:
\begin{equation}
\begin{aligned}
\Delta u(\boldsymbol{x}) + k^2 u(\boldsymbol{x})&=0, \qquad \boldsymbol{x}\in \Omega \\
\partial_{ n}u(\boldsymbol{x}) &= 0, \qquad \boldsymbol{x} \in \partial \Omega,
\end{aligned}\label{eq:neu_BVP}
\end{equation}
where~$n$ is the normal to~$\Omega$. For these systems, the modal decomposition~\eqref{eq:u_sep} holds 
with \eqref{eq:dir_modes} replaced by
\begin{equation}
    b_m^p(y_p) = \sqrt{\frac{2-\delta_{m1}}{d_p}}\cos\lp\frac{(m-1)\pi}{d_p} \lp y_p+\frac d2\rp\rp, \qquad \beta_m^p = \sqrt{k^2-\lp\frac{(m-1)\pi}{d_p}\rp^2}.\label{eq:neu_modes}
\end{equation}
Since the~$b_m^p$ are still orthonormal on~$\gamma_p$, the equivalent projections can still be used 
to find the coefficients~$c_{m,+}^p$ and Theorem~\ref{thm:uniq_prop} remains valid.

\section{Impedance-to-impedance maps}\label{sec:I2Imap}
In order to solve the system of equations~\eqref{eq:piece_helm}
and~\eqref{eq:piece_jump}, we will use the impedance-to-impedance
map for the truncated domain $\Omega_{0}$, which we now define. Recall
that the boundary of $\Omega_{0}$ is given by $\Gamma \cup \lp{\cup_{p=1}^{P} \gamma_{j}}\rp$.
We demonstrate the construction of solutions for the Dirichlet waveguides,
an identical procedure can be used for the Neumann waveguides. Consider
the following boundary value problem (BVP)
\begin{equation}\label{eq:comp_IBVP}
\begin{aligned}
(\Delta + k^2) v_{0}(\bx) &= 0\,, \quad \bx \in \Omega_{0} \, ,\\
v_{0}(\bx) &= 0\,, \quad \bx \in \Gamma \, ,\\
\frac{\partial v_{0}(\bx)}{\partial n} + i \eta v_{0}(\bx) &= f_{p}(\bx)\,, \quad \bx \in \gamma_{p} \,, 
\end{aligned}
\end{equation}
for $p=1,2,\ldots P$. 
\begin{lemma}
Suppose~$\eta$ is complex and~$\Re \eta\neq 0$.
For any $f_{p} \in H^{-1/2}(\gamma_{p})$ 
 \eqref{eq:comp_IBVP} has a unique solution $v_{0} \in H^1(\Omega_{0})$.     
\end{lemma}
\begin{proof}
    Uniqueness follows from the energy identity plus unique continuation; existence follows from a 
    G\r{a}rding inequality and the Fredholm alternative (see, for example,~\cite{mclean2000,gillman2015spectrally}).
\end{proof}
Once we have solved \eqref{eq:comp_IBVP}, we can compute the outgoing impedance 
data $g_{p} = \frac{\partial v_{0}}{\partial n} - i\eta v_{0}$ for $\bx \in \gamma_{p}$, $p =1,2\ldots P$. 
We define the impedance-to-impedance map for this setup, denoted $\Smat$, as the operator that maps 
the input vector $[f_{1}; f_{2}; \ldots f_{P}]$ to the output vector $[g_{1}; g_{2}; \ldots g_{P}]$. 
Both $f_p$ and $g_p$ contain contributions from the complete basis of propagating and evanescent modes. 
Consequently, their modal coefficient vectors belong to the infinite-dimensional Hilbert space $l^2$.

We now introduce a truncated impedance-to-impedance map, denoted by $\Strunc$. This map operates on a 
restricted input space, requiring that the input $f_p$ lies in the subspace $\text{span}\{b_m^p\}_{m=1}^{M_p}$. 
The map then projects the corresponding full output, $g_p$, onto this same subspace to produce the final 
output, $\tilde{g}_p$. In other words, we consider only the sine-series coefficients that correspond to 
propagating modes. More precisely, suppose that
\begin{equation}
f_{p}(y_{p}) = \sum_{m=1}^{M_{p}} \hat{f}_{m}^{p} b_{m}^{p}(y_{p}) \, ,\quad p=1,2,\ldots P \,,
\end{equation}
where $y_{p}$ as before is the transverse coordinate in port $p$. The functions $g_{1}, g_{2} \ldots g_{p}$ 
can also be expressed in a sine-series expansion since the homogeneous Dirichlet conditions imply 
that $\frac{\partial v_{0}}{\partial n} = \frac{\partial v_{0}}{\partial x_{p}} = 0$ at $y_{p} = \pm d_{p}/2$. 
Let $\hat{g}_{m}^{p}$ denote the sine-series coefficients of $g_{p}$, i.e.,
\begin{equation}
g_{p}(y_p) = \sum_{m=1}^{\infty} \hat{g}_{m}^{p} b_{m}^{p}(y_{p}) \,,
\end{equation}
for $p=1,2,\ldots P$. In this basis, the truncated impedance-to-impedance operator $\Strunc$ is represented 
by a matrix that maps the input coefficients $\hat{f}_{m}^{p}$ to the output coefficients $\hat{g}_{m}^{p}$, 
for $m=1,2,\ldots M_{p}$ and $p=1,2,\ldots P$. In other words, we have
\begin{equation}
\Strunc = \Pout\Smat\Pout,
\label{eq:PIP}
\end{equation}
where $\Pout$ is the projection operator onto the finite dimensional subspace containing only propagating
modes.

The key observation is that  $\Strunc$ serves as a compact mathematical representation of the interior 
domain $\Omega_0$, containing all the necessary information about its scattering characteristics.
To illustrate how $\Strunc$ can be used to solve the whole system~\eqref{eq:piece_out}, \eqref{eq:piece_helm} 
and ~\eqref{eq:piece_jump}, we consider a simplified two-port geometry ($P=2$) where each port supports a 
single propagating mode ($M_1=M_2=1$). The field is excited by an incoming wave in Port 1,  
$\uin_{1} = e^{-i \beta_{1}^{1}x_{1}} b_{1}^{1}(y_{1})$, while the incoming field in Port 2 is zero. 
We assume the port truncation length $L$ is sufficiently large such that evanescent modes have 
decayed ($|e^{-i \beta_{m}^{p} L}| < \varepsilon$ for all $m\ge 2$ and $p=1,2$). Under this assumption, 
$u_p = c^{p}_{+} e^{i \beta_{1}^{p}x_{p}} b_{1}^{p}(y_{p})+ O(\varepsilon)$, where the coefficients of the 
outgoing propagating modes $c_+^{1}, c_+^{2}$ are unknowns. Ignoring the $O(\varepsilon)$ term, the impedance 
data on the boundaries $\gamma_1$ and $\gamma_2$ are defined in terms of their first basis coefficients,
$\hat{f}^{1}_1$ and $\hat{f}^{2}_1$: 
\[ f_{1}(y_1) = \partial_{n_{1}} u_{0} + i\eta u_{0} = \hat{f}^{1}_1b_{1}^{1}(y_{1}) \qquad \text{and} \qquad f_{2}(y_2) = \partial_{n_{2}} u_{0} + i\eta u_{0} = \hat{f}^{2}_1b_{1}^{2}(y_{2}) \] 
Similarly, let $\hat{g}^{p}_{1}$ be the sine series coefficients of $\partial_{n_{p}} u _{0} - i\eta u_{0}$. 
The truncated operator $\Strunc$ then provides the linear map between these coefficients: 
\begin{equation} \Strunc \begin{bmatrix} \hat{f}_{1}^{1} \\ \hat{f}_{1}^{2} \end{bmatrix} = \begin{bmatrix} \hat{g}_{1}^{1} \\ \hat{g}_{1}^{2} \end{bmatrix}
\label{eq:Itruncsimple}
\end{equation}
Let $\vec{c}_{+} = [c^{1}_{+}; c^{2}_{+}]$, $\vec{\hat{f}} = [\hat{f}^{1}_{1}; \hat{f}^{2}_{1}]$,
and $\vec{\hat{g}} = [\hat{g}^{1}_{1}; \hat{g}^{2}_{1}]$.
Then the continuity conditions in \eqref{eq:piece_jump} are equivalent to 
\begin{equation}
    \partial_{n_{p}} u_{0} \pm i\eta u_{0}
    = (\partial_{n_{p}} u_{p} \pm i\eta u_{p}) + (\partial_{n_{p}} \uin_{p} \pm i\eta \uin_{p}), \qquad p=1,2.
\end{equation}
That is,
\begin{equation}
\begin{aligned}
\vec{\hat{f}} &= D_+ \vec{c}_{+} - 
\begin{bmatrix}
i(\beta_{1}^1 - \eta) \\
0
\end{bmatrix},\\
\vec{\hat{g}} &= D_- \vec{c}_{+} - 
\begin{bmatrix}
i(\beta_{1}^1 + \eta) \\
0
\end{bmatrix},
\end{aligned}
\label{eq:fgDc}
\end{equation}
where $D_{\pm}$ are $2 \times 2$ diagonal matrices given by 
\begin{equation}
D_{\pm} = \begin{bmatrix} 
i(\beta_{1}^{1} \pm \eta) & 0 \\
0 & i(\beta_{1}^{2} \pm \eta)
\end{bmatrix}\,.
\end{equation}
Combining \eqref{eq:Itruncsimple}
and \eqref{eq:fgDc}, we obtain 
\begin{equation}
\begin{bmatrix}
D_{+} & -I \\
D_{-} & -\Strunc
\end{bmatrix}
\begin{bmatrix}
\vec{c}_{+} \\ \vec{\hat{f}} 
\end{bmatrix}
= 
\begin{bmatrix}
i(\beta_{1}^1 - \eta) \\
0
\\
i(\beta_{1}^1 + \eta) \\
0 
\end{bmatrix} \,,
\end{equation}

This approach extends directly to a general waveguide circuit with $P$ ports, each supporting $M_p$ 
propagating modes. We define the global vector of impedance data coefficients as the concatenation 
of the coefficients from each port:
$\vec{\hat{f}} = [\vec{\hat{f}}^{1}; \vec{\hat{f}}^{2}; \ldots \vec{\hat{f}}^{P}]$ where each block 
is $\vec{\hat{f}}^{p} = [\hat{f}^{p}_{1}; \ldots \hat{f}^{p}_{M_{p}}]$. The vector of unknown outgoing 
propagating mode coefficients, $\vec{c}_{+}$, is structured in the same manner. 
Then the vectors for the impedance data ($\vec{\hat{f}}$), the unknown outgoing propagating mode 
coefficients ($\vec{c}_{+}$), and the incoming mode coefficients  ($\vec{c}_{-}$) satisfy the system of equations
\begin{equation}
\label{eq:final_i2i_glue}
\begin{bmatrix}
D_{+} & -I \\
D_{-} & -\Strunc
\end{bmatrix}
\begin{bmatrix}
\vec{c}_{+} \\
\vec{\hat{f}}
\end{bmatrix}
= \begin{bmatrix}
D_{-}\vec{c}_{-} \\ 
D_{+}\vec{c}_{-}
\end{bmatrix} \,, 
\end{equation}
where $D_{+}$ and $D_{-}$ are diagonal matrices 
with $\diag{D_{\pm}} = i[(\vec{\beta}^{1}\pm \eta); (\vec{\beta}^{2}\pm \eta); \ldots (\vec{\beta}^{P}\pm \eta)]$. 

In summary, our method reduces the waveguide problem~\eqref{eq:piece_out}, \eqref{eq:piece_helm} 
and~\eqref{eq:piece_jump} to a two-stage process. First, the matrix $\Strunc$ is computed to characterize 
the interior domain $\Omega_0$. Second, the complete solution is found by solving the linear 
system~\eqref{eq:final_i2i_glue} with an auxiliary unknown vector $\vec{\hat{f}}$. We note that solution 
to the linear system~\eqref{eq:final_i2i_glue} is only an 
approximation to~\eqref{eq:piece_out}, \eqref{eq:piece_helm}, and~\eqref{eq:piece_jump}, 
since the $O(\varepsilon)$ evanescent terms in $u_{p}$ were ignored in the derivation of~\eqref{eq:final_i2i_glue}.

\begin{theorem}\label{lem:i2i_inv} 
Suppose~$\eta$ in~\eqref{eq:comp_IBVP} satisfies~$\textrm{Re}(\eta)<0$. For any right-hand side, 
there exists a unique solution to the linear system~\eqref{eq:final_i2i_glue}.
\end{theorem}
\begin{proof}
Since \eqref{eq:final_i2i_glue} is a finite linear system, the uniqueness implies the existence. 
Thus, we only need to show the uniqueness, that is, if $\vec{c}_{-}$ is a zero vector, then 
both  $\vec{c}_{+}$ and  $\vec{\hat{f}}$ have to be the zero vectors.
Suppose that $[\vec{c}_{+},\vec{\hat{f}}]$ is a null vector of \eqref{eq:final_i2i_glue}. 
Suppose that $w_{0}$ is a solution to~\eqref{eq:comp_IBVP} with data
\begin{equation}
\frac{\partial w_{0}}{\partial n} + i\eta w_{0} \bigg|_{\gamma_{p}} = \sum_{m=1}^{M_{p}} \hat{f}^{p}_{m} b_{m}(y_{p}) \,. 
\end{equation}
The first block row of equations in~\eqref{eq:final_i2i_glue} imply that
\begin{equation}
\label{eq:datam}
\frac{\partial w_{0}}{\partial n} + i\eta w_{0} \bigg|_{\gamma_{p}} = \sum_{m=1}^{M_{p}} \hat{f}^{p}_{m} b_{m}(y_{p}) = \sum_{m=1}^{M_{p}} i(\beta_{m}^{p} + \eta) c_{m,+}^{p} b_{m}^{p}(y_{p}) \, .
\end{equation}
There exists a unique solution $w_{0}$ to~\eqref{eq:comp_IBVP} with this data, 
and let $a_{m}^{p}$ denote the coefficients of the outgoing impedance data, i.e.,
\begin{equation}
\frac{\partial w_{0}}{\partial n} - i\eta w_{0} \bigg|_{\gamma_{p}} = \sum_{m=1}^{\infty} a_{m}^{p} b_{m}^{p}(y_{p}) \,.
\end{equation}
Using the second block row of~\eqref{eq:final_i2i_glue}, it follows that 
\begin{equation}
a_{m}^{p} = i (\beta_{m}^{p} -\eta) c_{m,+}^{p}  \, ,\quad
m=1,2\ldots M_{p},\,p=1,2,\ldots P\,.
\end{equation}
and thus
\begin{equation}
\label{eq:datap}
\frac{\partial w_{0}}{\partial n} - i\eta w_{0} \bigg|_{\gamma_{p}} = \sum_{m=1}^{M_{p}} i(\beta_{m}^{p} -\eta) c_{m,+}^{p} b_{m}^{p}(y_{p}) + \sum_{m=M_{p}+1}^{\infty} a_{m}^{p} b_{m}^{p}(y_{p}) \,.
\end{equation}
Combining~\eqref{eq:datam} and~\eqref{eq:datap}, we get
\begin{equation}
\begin{aligned}
w_{0}|_{\gamma_{p}} &= \sum_{m=1}^{M_{p}} c_{m,+}^{p} b_{m}^{p}(y_{p}) - \sum_{m=M_{p}+1}^{\infty} \frac{a_{m}^{p}}{2i\eta} b_{m}^{p}(y_{p})  \\
\frac{\partial w_{0}}{\partial n}\bigg|_{\gamma_{p}} &=\sum_{m=1}^{M_{p}} i\beta_{m}^{p} c_{m,+}^{p} b_{m}^{p}(y_{p}) + \sum_{m=M_{p}+1}^{\infty} \frac{a_{m}^{p}}{2} b_{m}^{p}(y_{p})
\end{aligned}
\end{equation}
$p=1,2,\ldots P$.
Finally, applying Lemma~\ref{lem:opt_thm} to $w_{0}$ yields
\begin{equation}
\begin{aligned}
    0 = \textrm{Im}\int_{\partial \Omega_{0}} w_{0} \frac{\partial \overline{w_{0}}}{\partial n}
    &= \textrm{Im}\sum_{p=1}^{P} \int_{\gamma_{p}} w_{0} \frac{\partial \overline{w_{0}}}{\partial n} \\
    &= \textrm{Im}\sum_{p=1}^{P} \left(\sum_{m=1}^{M_{p}} c_{m,+}^{p} \cdot \overline{i\beta_{m}^{p} c_{m,+}^{p}}  - \sum_{m=M_{p}+1}^{\infty} \frac{a_{m}^{p}}{2i\eta} \cdot \frac{\overline{a_{m}^{p}}}{2}\right)\\
    &= -\sum_{p=1}^{P} \left(\sum_{m=1}^{M_{p}}|c_{m,+}^{p}|^2 - \frac{1}{4\textrm{Re}(\eta)}\sum_{m=M_{p}+1}^{\infty} |a_{m}^{p}|^2 \right) \,.
\end{aligned}
\end{equation}
Since~$\textrm{Re}(\eta)<0$, the result follows.
\end{proof}

\begin{remark}
From~\eqref{eq:final_i2i_glue}, it is clear that if~$\Strunc$ is known, the scattering matrix 
can be computed as follows
\begin{equation}\label{eq:scattering_mat}
    \begin{bmatrix}
        \cS\\
        \cF
    \end{bmatrix} = \begin{bmatrix}
D_{+} & -I \\
D_{-} & -\Strunc
\end{bmatrix}^{-1} \begin{bmatrix}
D_{-} \\ 
D_{+}
\end{bmatrix},
\end{equation}
where~$\cF$ is the map from~$\vec{c}_-$ to the impedance data~$\vec{\hat{f}}$. 
\end{remark}

\begin{remark}
In order for the scattering matrix computed via \eqref{eq:scattering_mat} to be close to its true value 
for the original problem, $u_p$ should be well approximated by the outgoing propagating 
modes ~$\sum_{m=1}^{M_p}c_{+,m}^p e^{i\beta_m^px_p}b_m^p(y_p)$. This can always be achieved 
because we can place the external ports at a truncation distance~$L$ in Fig.~\ref{fig:domain_def} large 
enough from the interacting regime that all evanescent modes~$e^{-i\beta_{m}^px_p}b_{m}^p(y_p)$ ($m> M_p$) 
have decayed to below a given tolerance. 
\end{remark}

\subsection[Computing impedance to impedance maps via domain decomposition]{Computing $\Strunc$ via domain decomposition}\label{sec:merge2}
For simple geometries, one can find $\Strunc$ by directly discretizing and solving~\eqref{eq:comp_IBVP}. 
However, the computational cost of solving~\eqref{eq:comp_IBVP} grows rapidly with the size of the waveguide circuit.

It is possible to accelerate the construction of $\Strunc$ based on a domain decomposition approach. This 
can be done by partitioning the truncated domains into many modular components, constructing the 
impedance-to-impedance operators for the individual components, and imposing continuity of $u$ 
and $\frac{\partial u}{\partial n}$ at the common edges of the individual components. 

To illustrate this approach, suppose that the region $\Omega_{0}$ with two ports $\gamma_{1}, \gamma_{2}$ 
is partitioned into two components $\tilde\Omega_{1}$ and $\tilde\Omega_{2}$, 
i.e., $\overline{\Omega_{0}} = \overline{\tilde\Omega_{1}} \cup \overline{\tilde\Omega_{2}}$, with both 
components having one port each. Let $\gamma_{12}$ denote their common boundary, 
i.e., $\gamma_{12} = \partial \tilde \Omega_{1} \cap \partial\tilde \Omega_{2}$, see 
Figure~\ref{fig:partition}. We assume that the normal to $\gamma_{12}$ is pointing away 
from $\tilde \Omega_{1}$. Let $\Strunc_{j}$ denote the impedance to impedance operators 
for $\tilde \Omega_{j}$, $j=1,2$ which we write in the following block $2\times 2$ form corresponding 
to impedance data on $\gamma_{1}$ and $\gamma_{12}$ for $\tilde \Omega_{1}$, and $\gamma_{2}$ 
and $\gamma_{12}$ for $\tilde \Omega_{2}$, i.e.,
\begin{equation}
\Strunc_{1} = \begin{bmatrix}
A^{(1)}_{\gamma_{1}, \gamma_{1}} & A^{(1)}_{\gamma_{1}, \gamma_{12}} \\
A^{(1)}_{\gamma_{12}, \gamma_{1}} & A^{(1)}_{\gamma_{12}, \gamma_{12}} 
\end{bmatrix} \, , \quad 
\Strunc_{2} = \begin{bmatrix}
A^{(2)}_{\gamma_{2}, \gamma_{2}} & A^{(2)}_{\gamma_{2}, \gamma_{12}} \\
A^{(2)}_{\gamma_{12}, \gamma_{2}} & A^{(2)}_{\gamma_{12}, \gamma_{12}} 
\end{bmatrix} \, .
\end{equation}
\begin{figure}
\centering
\includegraphics[width=\linewidth]{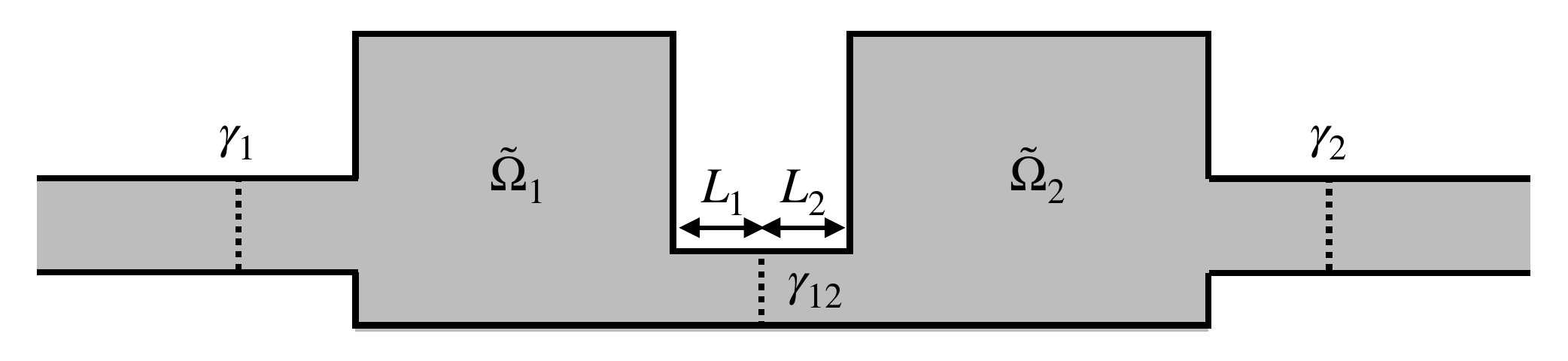}
\caption{An example of a component divided into components~$\tilde\Omega_1$ 
and~$\tilde\Omega_2$ by the curve~$\gamma_{12}$. The figure also shows the 
definition of the lengths~$L_1$ and~$L_2$.}
\label{fig:partition}
\end{figure}
Given the impedance data vectors
$\vec{\hat{f}} = [\vec{\hat{f}}^{1}; \vec{\hat{f}}^{2}]$
and $\vec{\hat{g}} = [\vec{\hat{g}}^{1}; \vec{\hat{g}}^{2}]$
on $\gamma_1$ and $\gamma_2$, our goal is to find the matrix  $\Strunc$
such that $\vec{\hat{f}} = \Strunc\vec{\hat{g}}$.  Let $h_{\pm} =  \lp\frac{\partial u}{\partial n} \pm i\eta u\rp \big|_{\gamma_{12}}$ denote the impedance data on the interface $\gamma_{12}$ and $\hat{h}_{\pm}$ are the coefficient vectors of $h_{\pm}$ in the modal basis inherited from~$\partial\tilde\Omega_1$. Then by definition, we have 
\begin{equation}
\begin{bmatrix} \vec{\hat{g}}^{1} \\
\hat{h}_{-}
\end{bmatrix}
= \Strunc_{1} \begin{bmatrix}
\vec{\hat{f}}^{1} \\ 
\hat{h}_{+}
\end{bmatrix} \, ,
\quad
\begin{bmatrix}
\vec{\hat{g}}^{2} \\
- D \hat{h}_{+}
\end{bmatrix} = 
\Strunc_{2} \begin{bmatrix}
\vec{\hat{f}}^{2} \\
-D\hat{h}_{-}
\end{bmatrix} \, ,
\end{equation}
where $D$ is the diagonal matrix with entries~$\pm 1$ depending on the symmetry of the corresponding basis function.
Note that negative signs in the above equation relating to $\Strunc_{2}$ account for the fact that 
the normal to $\gamma_{12}$ points into $\tilde\Omega_{2}$. 
These equations can be combined to give the following block system
\begin{equation} \label{eq:merge_sys}
    \begin{bmatrix}
        A^{(1)}_{\gamma_{1}, \gamma_{1}}  & 0 & A^{(1)}_{\gamma_{1}, \gamma_{12}} & 0 \\
    0 &A^{(2)}_{\gamma_{2}, \gamma_{2}} &  0& -A^{(2)}_{\gamma_{2}, \gamma_{12}} \\
        A^{(1)}_{\gamma_{12}, \gamma_{1}} & 0 & A^{(1)}_{\gamma_{12}, \gamma_{12}} & -I\\
        0 & A^{(2)}_{\gamma_{12}, \gamma_{2}} & D &  -A^{(2)}_{\gamma_{12}, \gamma_{12}} D
    \end{bmatrix}\begin{bmatrix}
        \vec{\hat{f}}^{1}\\\vec{\hat{f}}^{2}\\ \hat{h}_+ \\ \hat{h}_-
    \end{bmatrix}=\begin{bmatrix}
        \vec{\hat{g}}^{1}\\
        \vec{\hat{g}}^{2}\\0\\ 0
    \end{bmatrix}.
\end{equation}

We can use a Schur complement to eliminate $\hat{h}_{\pm}$ and find
\begin{equation}
\begin{bmatrix}
\vec{\hat{g}}^{1}\\
\vec{\hat{g}}^{2}
\end{bmatrix} = \Strunc
\begin{bmatrix}
\vec{\hat{f}}^{1}\\
\vec{\hat{f}}^{2}
\end{bmatrix} \,,
\end{equation}
where
\begin{equation}\label{eq:merged_S}
\begin{aligned}
\Strunc &= \begin{bmatrix}
A_{\gamma_{1}, \gamma_{1}}^{(1)} & 0 \\
0 & A_{\gamma_{2}, \gamma_{2}}^{(2)}
\end{bmatrix} \\
&- 
\begin{bmatrix}
A^{(1)}_{\gamma_{1}, \gamma_{12}} & 0 \\
0 & -A^{(2)}_{\gamma_{2}, \gamma_{12}}
\end{bmatrix}
\begin{bmatrix} 
A^{(1)}_{\gamma_{12}, \gamma_{12}} & -I \\
D & -A^{(2)}_{\gamma_{12}, \gamma_{12}}D
\end{bmatrix} ^{-1}
\begin{bmatrix}
A^{(1)}_{\gamma_{12}, \gamma_{1}} & 0 \\
0 & A^{(2)}_{\gamma_{12}, \gamma_{2}} 
\end{bmatrix}
\, .
\end{aligned}
\end{equation}

\subsection{Selection of modes at the interface}\label{sec:merge_acc}
The divide-and-conquer scheme, using truncated impedance-to-impedance
maps, can effectively handle closely connected sub-components. When
sub-components are close, retaining only propagating modes at the
interface $\gamma_{12}$ results in a significant loss of accuracy.
It is therefore necessary to include a sufficient number of evanescent
modes, $b_{m}^{12}$ for $m>M_{12}$, in the impedance-to-impedance
maps $\Strunc_1$ and $\Strunc_2$. These modes, which would decay
in a semi-infinite port, are essential for resolving near-field interactions.
The merging formulas~\eqref{eq:merge_sys}--\eqref{eq:merged_S}
in Section~\ref{sec:merge2} apply directly to this expanded modal
basis. If these were ports instead of continuing to another subdomain,
these modes would correspond to evanescent modes which would decay
exponentially away from the interface $\gamma_{12}$. Based on this
heuristic, we expect the error in~\eqref{eq:merged_S} to decay
like
\begin{equation}
    e^{i\beta_{M+1}^{12} \min(L_1,L_2)} \sim  e^{-\frac{\pi}{d_{12}}\min(L_1,L_2) M},
\end{equation}
where~$M \ge M_{12}$ is the total number of modes included in~$\hat{h}_{\pm}$.  To verify the merging 
formula and the heuristic above, consider the domain in~\figref{fig:merge_soln},
which is constructed from two simple components with $M_{12}=1$. In~\figref{fig:merge_err}, 
we plot the error as a function of the channel length, $L$, i.e., the length of the rectangular 
section between the two components, and the number of terms $M$ included in the representation 
on $\gamma_{12}$. For both Dirichlet and Neumann waveguides, the error decreases exponentially 
with respect to $L$, and the rate of this decay increases as more modes are added.

\begin{remark}
In \cite{sugita2023thesis}, a closely related approach is studied, with Dirichlet conditions 
imposed at the interfaces. As in \cite{wang2025scattering}, that work assumes the components 
are connected by sufficiently long straight channels so that evanescent modes decay below a prescribed 
tolerance and can be neglected at the interface.
\end{remark}

\begin{figure}
    \centering
    \includegraphics[width=0.5\linewidth]{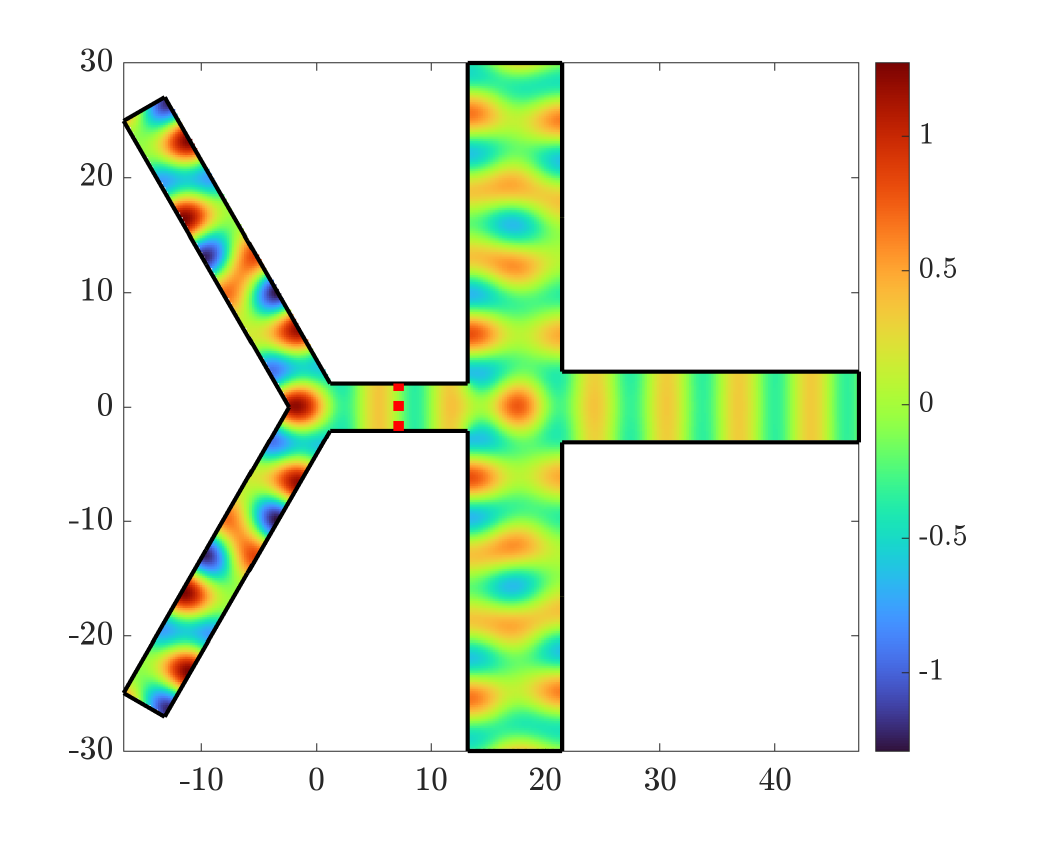}
    \caption{The real part of an example solution in geometry used to test the 
    accuracy of~\eqref{eq:merged_S}. 
    The red dashed line indicates the line separating the components, $\gamma_{12}$.}
    \label{fig:merge_soln}
\end{figure}

\begin{figure}
    \centering
    \includegraphics[width=0.9\linewidth]{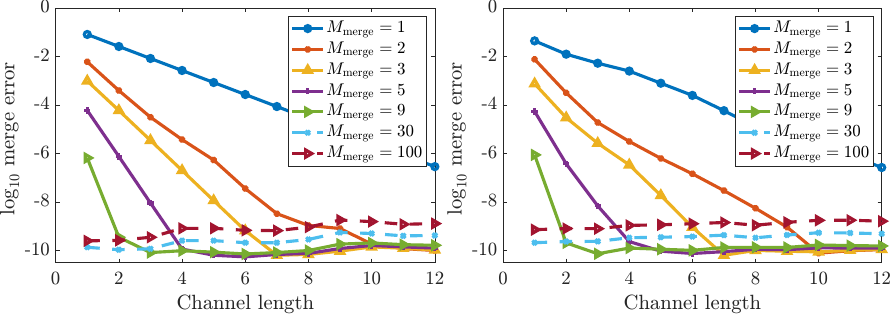}
    \caption{Error in merging impedance-to-impedance maps of two sub-domains separated by channel length $L$. 
    The left and right figures show the errors for Dirichlet and Neumann boundary conditions respectively.}
    \label{fig:merge_err}
\end{figure}

\subsection{A graph-based method for merging multiple  components}\label{sec:more_comp}
If the region~$\Omega_0$ is partitioned into the subcomponents~$\tilde\Omega_1,\ldots \tilde\Omega_N$, we can 
extend the merging formulas in Section~\ref{sec:merge2} to build~$\Strunc$ for the whole domain. In this case, 
we can think of the subcomponents as nodes on a graph. The nodes~$\tilde\Omega_j$ and~$\tilde\Omega_l$ are 
said to be connected by an edge if~$\gamma_{jl}=\partial\tilde\Omega_j\cap \partial\tilde\Omega_l$ is nonempty. 
If the local impedance-to-impedance maps~$\Strunc_j$ are known, we can construct~$\Strunc$ through an analogue 
of~\eqref{eq:merge_sys} and \eqref{eq:merged_S}.

To do this, we introduce the auxiliary coefficient vectors~$\hat{h}^{j,l}_{\pm}$ which represent the 
impedance data $\partial_{n_j} u \pm i\eta u\big|_{\gamma_{jl}}$ on the interfaces separating~$\tilde \Omega_j$ 
and~$\tilde \Omega_l$. We then construct a sparse block system of the form~\eqref{eq:merge_sys}. If the 
node~$\tilde\Omega_j$ is connected to~$P_j$ other nodes, then its truncated impedance-to-impedance 
map~$\Strunc_j$ will have~$P_j \times P_j$ blocks. The nonzero blocks in the row of the sparse block 
system corresponding to~$\hat{h}^{j,l}_{\pm}$ will consist of~$P_j$ sub-blocks of $\Strunc_j$ and one 
diagonal matrix. If the external port~$\Omega_p$ is connected to the component~$\tilde\Omega_{j_p}$, then 
the row corresponding to the impedance data~$\vec{\hat{f}}^p$ will have~$P_{j_p}$ non-zero blocks.
Matrices with this graph-based sparsity patterns have been well studied and there exist many algorithms 
capable of efficiently solving such systems, see \cite{davis2004algorithm,zhu2025recursive}, for example. 
In practice the number of components and the number of modes per edge are small enough that the system 
can be solved quickly using any sparse matrix solver. 

\begin{remark}
Some waveguide circuits may include long straight segments. If these are long enough, then the 
computational cost of discretizing these segments can be prohibitively expensive. To avoid this problem, 
we can replace the center of this segment by an extra circuit element. For this circuit element the 
impedance-to-impedance map can be found analytically, which removes the need to discretize these long segments.
\end{remark}

\section{Computing impedance maps via boundary integral equations}\label{sec:bie}
We use a boundary integral equation (BIE) formulation to solve the
BVP~\eqref{eq:comp_IBVP} and compute the truncated impedance-to-impedance
map for each sub-component. Truncating the computational domains
for waveguide systems typically results in many right-angled corners.
It is well-known that solutions to boundary integral equations can
develop singularities in the vicinity of these corners and at junctions
with mixed boundary conditions. A related issue when solving integral
equations with mixed boundary conditions is that using standard representations
associated with each boundary condition restricted to the relevant
boundaries requires evaluation of nearly hypersingular integrals
and the resulting integral operators may not be Fredholm. 

To illustrate this issue, consider the case of Dirichlet waveguides. Let~$G_k(r)=i H_0^{(1)}(kr)/4$ 
denote the fundamental solution to the 2D Helmholtz equation, where $H_0^{(1)}$
is the zeroth-order Hankel function of the first kind. Let $\cS_{\tilde\gamma}$ and $\cD_{\tilde\gamma}$ 
denote the single and double layer operators given by
\begin{equation}
      \cS_{\tilde\gamma}[\sigma](\bm{x})
  := \int_{\tilde\gamma} G_{k}(\bm{x},\bm{y}) \sigma(\bm{y}) ds_{\bm{y}}
      \quad\text{and}\quad 
      \cD_{\tilde\gamma}[\sigma](\bm{x})
  := \int_{\tilde\gamma} \PD{G_k}{n_{y}}(\bm{x},\bm{y}) \sigma(\bm{y}) dS_{\bm{y}} \,.
\end{equation}
Here~$\tilde\gamma$ is any curve in~$\bbR^2$. While the standard method for an interior Dirichlet 
problem uses a combined field representation $\cD + ik\cS/2$, the interior impedance problem is 
typically solved using only the single layer representation. Thus, it might be tempting to solve 
the mixed boundary value problem by representing the solution as
\begin{equation}
    u = 2\cS_{\cup_p\gamma_p}[\mu]   -2\lp\cD_{\Gamma} + \frac{ik}{2} \cS_{\Gamma}\rp [\sigma] ,
\end{equation}
where the constants $\pm 2$ are chosen such that the jumps on the boundary lead to the identity matrix. 
Unfortunately, this representation will be hypersingular due to the mixed boundary conditions and 
the corner between~$\Gamma$ and~$\gamma_p$. 

Instead, we take inspiration from the method of images and let~$\Gamma_r$ be the portion of~$\Gamma$ 
within a distance~$r$ of~$\gamma_p$. 
Let $\tilde\Gamma_r$ be the reflection of each piece of $\Gamma_r$ about the closest $\gamma_p$,
and let $\mathcal R:\Gamma_r\to\tilde\Gamma_r$ denote this reflection.
Set $\sigma_r := \restr{\sigma}{\Gamma_r}$ and
$\tilde\sigma_r := \sigma_r\circ\mathcal R^{-1}$ on $\tilde\Gamma_r$.
In order to remove the dominant corner singularity, we add the layer potential
$\cD_{\tilde\Gamma_r}$ with density $\tilde\sigma_r$ to our representation:
\begin{equation}\label{eq:u_rep}
    u = 2\cS_{\cup_p\gamma_p}[\mu]   -2\lp\cD_{\Gamma} + \frac{ik}{2} \cS_{\Gamma}\rp [\sigma] -2\cD_{\tilde\Gamma_r}[\tilde\sigma_r] \,.
\end{equation}
This representation is analogous to using the half-space Dirichlet double layer potential for the 
Helmholtz equation on $\Gamma_{r}$ instead of $\cD_{k}$.

The standard jump relations and symmetry of $\Gamma_r$ and~$\tilde\Gamma_r$ give that~$u$ 
solves~\eqref{eq:comp_IBVP} provided~$\sigma$ and~$\mu$ solve
\begin{equation}\label{eq:CFIE}
    \begin{split}
\mu + (\partial_{\bs n} +i\eta)\lp 2\cS_{\cup_p\gamma_p}[\mu]  -2\lp\cD_{\Gamma} + \frac{ik}{2} \cS_{\Gamma}\rp [\sigma]-2\cD_{\tilde\Gamma_r}[\tilde\sigma_r]\rp =f &\text{ on } \cup_p\gamma_p \\
        \sigma + 2\cS_{\cup_p\gamma_p}[ \mu]  -2\lp\cD_{\Gamma} + \frac{ik}{2} \cS_{\Gamma}\rp [\sigma] -2\cD_{\tilde\Gamma_r}[\tilde\sigma_r]=0 &\text{ on } \Gamma
    \end{split}
\end{equation}

\begin{figure}
    \centering
    \includegraphics[width=0.9\linewidth]{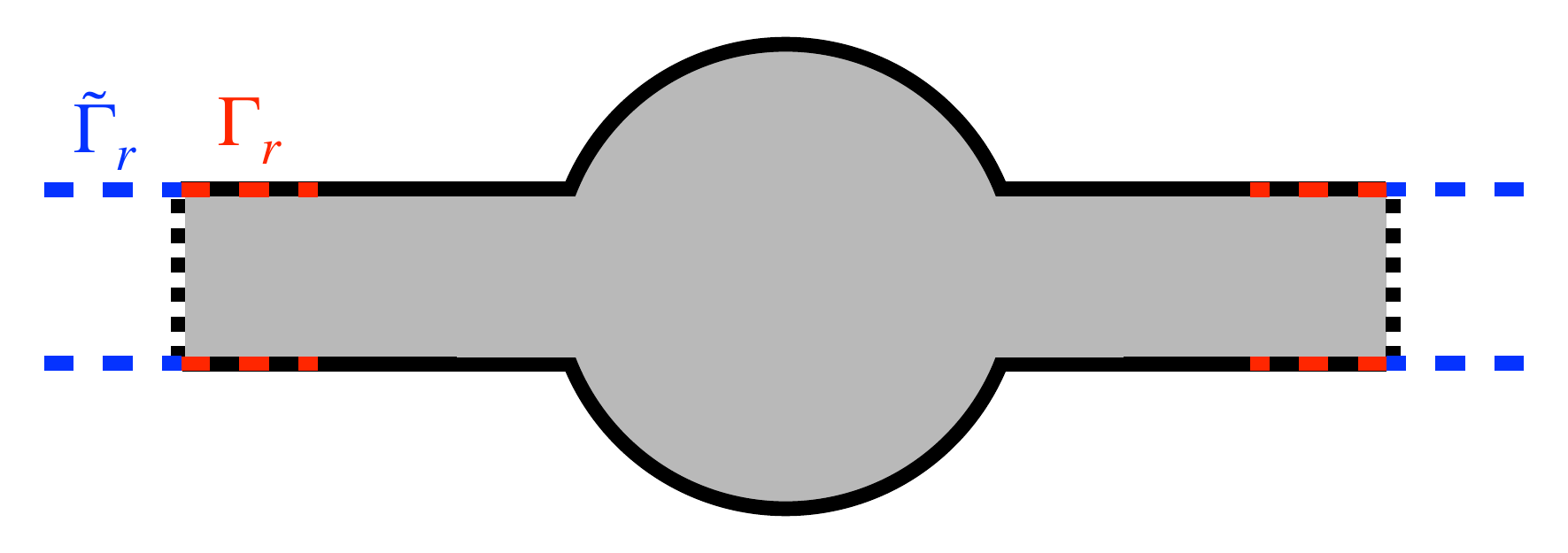}
    \caption{An illustration of our image curves. 
    The dashed red lines indicate the portion of~$\Gamma$ within a distance~$r$ of~$\gamma$ and 
    the dashed blue lines are the reflection of those segments about the line segments~$\gamma_p$.}
    \label{fig:image_def}
\end{figure}

\begin{theorem}
    If~$\Gamma$ is piecewise smooth and Lipschitz and~$|\eta| <(\sqrt{2}-1)/2$, then the integral operator 
    in~\eqref{eq:CFIE} is Fredholm index zero.
\end{theorem}
\begin{proof}
The only terms in \eqref{eq:CFIE} with potentially hypersingular behavior at the corner are
\begin{equation}\label{eq:hyper_dir}
    \partial_{\bs n} \lp-2\cD_\Gamma[\sigma] - 2\cD_{\tilde\Gamma_r}[\tilde\sigma_r]\rp.
\end{equation}
To analyze the singularity, we further split~$\Gamma_r$ and~$\tilde\Gamma_r$ into~$\Gamma_{r,p}$ 
and~$\tilde\Gamma_{r,p}$, their connected components touching~$\gamma_p$. 
Since $G_k$ is smooth away from the diagonal, the only potentially singular contribution in~\eqref{eq:hyper_dir} is
\begin{equation}
     \partial_{\bs n} \lp-2\cD_{\Gamma_{r,p}}[\sigma] - 2\cD_{\tilde\Gamma_{r,p}}[\tilde\sigma_r]\rp\big|_{\gamma_p}.
\end{equation}
By construction, the quantity in parentheses is even with respect to reflection across $\gamma_p$, 
so its normal derivative is odd and therefore vanishes on $\gamma_p$.
The remaining non-compact terms in~\eqref{eq:CFIE} are the double layer and the normal derivative of the 
single layer at the corners of~$\partial\Omega_0$. 

To prove the Fredholm structure, we note that it was proved in~\cite{shelepov1969index} 
(see~\cite{wendland2009double} for a summary) that for all~$\epsilon>0$ there exists a compact 
operator~$\tilde\cD$ such that
\begin{equation}
    \left\| \cD_{\Gamma \cup \lp\cup_p\gamma_p \rp} - \tilde \cD\right\|_{L^2} \leq \frac{1}2\sin \frac{\pi}{2\cdot 2} +\frac\epsilon2 = \frac{1}{2\sqrt{2}}+\frac\epsilon2\, .
\end{equation}
There also exists a compact operator~$\tilde\cS'$ such 
that~$ \| \cS_{\Gamma \cup \lp\cup_p\gamma_p \rp}' - \tilde \cS'\|_{L^2} \leq \frac{1}{2\sqrt{2}}+\frac\epsilon2$. 
Thus, after multiplying by some cutoff functions, the operator on the left hand side \eqref{eq:CFIE} can 
be written as an identity operator, plus a compact operator, plus an operator of norm~$1/\sqrt{2}+\epsilon$ 
and an operator with norm~$2\eta(1/\sqrt{2}+\epsilon)$. A Neumann series argument thus gives that the operator 
in~\eqref{eq:CFIE} is of the form invertible plus compact provided~$|\eta|< (\sqrt{2}-1)/2$.
\end{proof}

For the Neumann waveguides, the mixed boundary conditions are relatively easier to handle, 
as it can be reformulated as an impedance problem on the whole boundary but with a piecewise constant  
impedance function. Using a standard single layer representation (see \cite{Book-Kress-Colton-2013}) results 
in a Fredholm system of equations for the unknown densities in this case. 

\section{Numerical algorithm}\label{sec:alg}
\subsection{Discretization}
The boundary~$\Gamma$ is partitioned into 16th-order Gauss--Legendre
panels, and \eqref{eq:CFIE} (or its Neumann counterpart) is discretized
via a modified Nystr{\"o}m scheme. In this scheme, the unknown density~$\sigma$
is represented by its values at the quadrature nodes, and the equation
is enforced by collocation at those nodes.

To evaluate the integral operators in~\eqref{eq:u_rep} and~\eqref{eq:CFIE}
at a target~$\bx$, we split the boundary integral into contributions
from the self panel, panels near~$\bx$, and the remaining (far)
panels. For the singular self-panel and nearly singular near-panel
contributions we use generalized Gaussian quadrature~\cite{bremer2010nonlinear},
whereas the smooth far-panel contributions are evaluated with standard
Gauss–Legendre quadrature. We note that there are several other high-order
quadrature schemes for singular and nearly singular layer potentials,
including quadrature by expansion (QBX) and its variants, as well
as kernel-split methods; see, e.g.,~\cite{klockner2013jcp,helsing2008jcp}.
We treat corner singularities using recursively compressed inverse
preconditioning (RCIP)~\cite{JCP-2008-Helsing, RCIP-Tutorial, SISC-2018-Helsing-Jiang, JCP-2022-Helsing-Jiang}.
Although the layer-potential densities are singular at corners, RCIP
requires only quadrature rules that resolve the singularities of
the integral kernels.

\subsection{Algorithm and complexity analysis}
Let $\EPS>0$ denote the prescribed tolerance. We summarize our approach for computing $\mathcal{S}$ 
in Algorithm~\ref{alg:1}.
\begin{algorithm}[!htbp]
\caption{Fast computation of the scattering matrix $\mathcal{S}$}
\label{alg:1}
\mbox{}
\begin{enumerate}
\item Subdivide the domain~$\Omega_0$ into~$N$ components~$\tilde\Omega_i$ and let divide each channel 
connecting components halfway between adjacent components. 
\item Set the number of modes in port~$p$ of~$\tilde\Omega_i$, denoted by~$\widetilde M_{i,p}$, in each 
port to be the largest value such that~$\left|e^{i\beta_{\widetilde M_{i,p}}L_{i,p}}\right|>\EPS$.
\item Discretize each boundary~$\partial\tilde\Omega_i$, ensuring that there are enough points on 
each~$\gamma_{i,p}$ to resolve the mode~$b_{\widetilde M_{i,p}}$. Then construct the system matrix for~\eqref{eq:CFIE}, and its LU factorization.
\item Compute $\Strunc_i$ for each component~$\tilde\Omega_i$ by solving~\eqref{eq:comp_IBVP} once 
for each mode in each channel connected to the component. This step requires $\widetilde M_{i} = \sum_{p=1}^{P_i}\widetilde M_{i,p}$ solves.
\item Construct the linear system equivalent to~\eqref{eq:merge_sys} that enforces continuity of 
the solution and its normal derivative in every interface (see Section~\ref{sec:more_comp}) and use a sparse 
solver to find the Schur complement and compute the truncated impedance-to-impedance map for~$\Omega_0$ 
(see \eqref{eq:merged_S}).
\item Solve~\eqref{eq:scattering_mat} to compute the scattering matrix for the domain~$\Omega_0$.
\end{enumerate}
\end{algorithm}

In order to analyze the computational complexity of the algorithm,
for simplicity, suppose that each of the~$N$ components are distinct
and so the local impedance-to-impedance maps~$\Strunc_i$ can't
be reused. Suppose also that all~$\tilde\Omega_i$'s are comparable
in size and thus discretized using~$n$ points each. Finally, suppose
that $\widetilde M_{i}=\widetilde M$ is independent of $i$.

Under these assumptions, the most costly parts of the algorithm are step 3 with a cost of~$O(N n^3)$ and 
step 4 with a cost of~$O(N n^2 \widetilde M)$. The cost of building the linear system in step 5 
is~$O(N \widetilde M)$ and the cost of solving the system with a modern sparse solver such 
as~\cite{zhu2025recursive} will be~$O(N \widetilde M)$. In practice, $\widetilde M \ll n $, so that 
any sparse solver makes the cost of step 5 negligible compared to the cost of steps 3 and 4. Finally the 
cost of step 6 is~$O(1)$ and so the total cost of our algorithm is~$O(N n^3+N n^2 \widetilde M)$. 

Before proceeding, we discuss the criteria for identifying the components~$\tilde\Omega_i$. If two 
components are connected by a very short channel, then the number of relevant modes~$\widetilde M_{i,p}$ will 
be large for some~$p$. In the above, we noted that the cost of computing~$\Strunc_i$ for a single component 
will be $n^3+ n^2 \widetilde M.$ Thus, it is typically more efficient to treat adjacent subdomains as separate 
components unless the connecting channel is so short that 
$\widetilde M_{i,p}$ is comparable to the number of discretization nodes for either element. 

With this choice, each domain $\tilde\Omega_{i}$ is $O(1)$ wavelength in size, and for the examples considered 
in this work $n$ is $O(10^{3})$. For such small $n$, dense linear algebra for steps 3–4 typically outperforms 
FMM-accelerated iterative solvers or fast direct methods, because the latter incur large constant factors; 
this prefactor penalty is especially severe at high accuracies. For larger values of $n$, the dense linear 
methods can be easily be replaced by fast direct solvers (these would perform better than FMM accelerated 
iterative solvers since solutions for $\widetilde M$ different boundary data are required).

\begin{remark}
The cost of RCIP compression is dominated by kernel evaluations within
the recursion. Because our domains contain many corners, we fix the
lengths of the panels adjacent to each corner and interpolate the
resulting compressed RCIP blocks using a piecewise-Chebyshev expansion
in the corner opening angle. 
\end{remark}

\section{Numerical Examples}\label{sec:num_exam}
Algorithm~\ref{alg:1} has been implemented using \texttt{chunkie} --- an excellent MATLAB integral equation 
toolbox~\cite{chunkIE}. All numerical experiments in this section were run on a MacBook Pro with an M2 Max 
Chip (12 cores). The tolerances were set to $10^{-14}$, unless otherwise specified, and the Helmholtz 
wavenumber was fixed at $k=1$.

\subsection{Accuracy test}
In order to test the accuracy of the solver for~\eqref{eq:comp_IBVP}, consider the computational domain 
with a single output port in Fig.~\ref{fig:analytic_err}. Suppose that the boundary data in~\eqref{eq:comp_IBVP} 
corresponds to the known solution
\begin{equation}
     u(\bx) = G_k(\bx,\bx_0)
\end{equation}
with ~$\bx_0$ located outside the domain. 
In~\figref{fig:analytic_err}, we compare the numerical solution (for~$\bx_0=(40,20)$) with the exact field. 
The computed solution is correct to about ten digits of accuracy throughout the domain and away from the corners, 
for both Dirichlet and Neumann boundary conditions. 

\begin{figure}
    \centering
    \includegraphics[width=\linewidth]{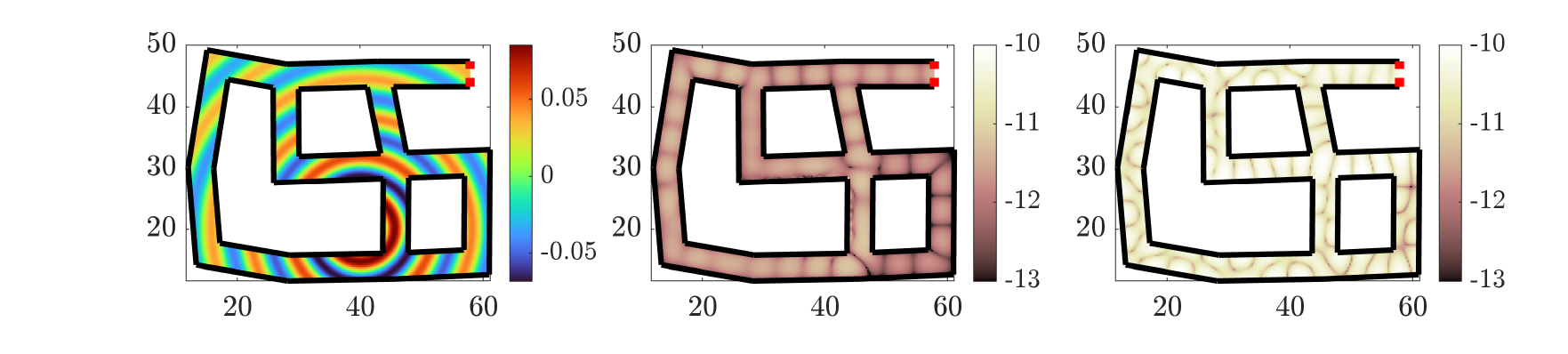}
    \caption{The results of the analytic solution test for our solvers for Dirichlet or Neumann waveguide 
    components. In the left figure, we plot the real part of the reference solution. In the middle and right 
    figures we plot  $\log_{10}$ of the absolute error for the Dirichlet and Neumann solvers respectively.}
    \label{fig:analytic_err}
\end{figure}

\subsection{Timing results and comparison with other fast solvers}
We now study how the computational cost of Algorithm~\ref{alg:1} scales with the number of 
components $N$. Consider a perturbed lattice of $6 \times n_{x}$ points with grid spacing~$15$ and 
perturbations drawn uniformly from $[-2.25,2.25]\times[-2.25,2.25]$. Consider a graph where $80\%$ of 
the edges corresponding to neighboring points on the unperturbed lattice are connected. The individual 
components $\tilde\Omega_{i}$ are defined as the union of straight channels for all edges on the graph 
at vertex $i$, and the device $\Omega_{0} = \cup_{i=1} \tilde\Omega_{i}$. The width of the channel 
to $d=\pi + 1$, so that there is one propagating mode for the Dirichlet case, and 2 propagating modes 
for the Neumann case. In this construction, we also enforce that the full device has exactly $3$ 
external ports for all values of $n_{x}$.

\figref{fig:merge_quilt} illustrates an example with $n_{x} = 19$. The number of components scale 
linearly with $n_{x}$ while $n$ and $\widetilde{M}$ remain approximately the same. For the example 
with $n_{x} = 19$, the average $n$ was $2717$ and the average $\widetilde{M}$ was $19$. The maximum number 
of modes in any channel $\max_{i,p} \widetilde M_{i,p}$ was $8$ indicating the use of evanescent modes 
in the merge step were necessary to achieve the desired tolerance. 

We compare the performance of Algorithm~\ref{alg:1} (labeled “Merged” in \figref{fig:merge_timing}) 
against three fast algorithms that differ only in how they compute $\Strunc$. The first, 
labeled “FMM,” solves~\eqref{eq:comp_IBVP} to a tolerance of~$10^{-10}$ with GMRES~\cite{GMRES}, 
accelerating matrix–vector products using the fast multipole method~\cite{JCP-1987-fmm} as implemented 
in \texttt{fmm2d}~\cite{fmm2d}. The second, labeled “FLAM,” solves~\eqref{eq:comp_IBVP} using the 
recursive-skeletonization fast direct solver~\cite{SISC-Ho-2012} in \texttt{FLAM}~\cite{ho2020flam}. 
The third, labeled “FMM + FLAM,” employs a low-accuracy recursive-skeletonization approximate inverse 
as a preconditioner for the FMM-accelerated GMRES solve; the approximate inverse was constructed to 
tolerance $10^{-3}$.

The total number of discretization points $n_{\mathrm{pts}}$ on $\Omega_{0}$ also scales linearly 
with $N$, but is a constant factor smaller than $Nn$; in all examples this factor satisfies $\le 2$. 
This reduction occurs because discretizing $\Omega_{0}$ does not require discretizing the subdomain 
boundaries $\gamma_{p,i}\subset\tilde\Omega_{i}$. The runtimes of all four methods as a function 
of $n_{\mathrm{pts}}$ are plotted in \figref{fig:merge_timing}. Algorithm~\ref{alg:1} is faster than 
all three fast alternatives described above, and its CPU time scales linearly with $n_{\mathrm{pts}}$. 
Although, asymptotically, the cost of solving the sparse linear system~\eqref{eq:merge_sys} would dominate, 
the prefactor is so small that it is negligible in these tests: in the largest Dirichlet waveguide example, 
constructing $\Strunc_i$ for $i=1,2,\ldots,N$ took $104$ seconds, whereas the sparse solve took only $0.025$ seconds.

The CPU time for the other three methods appears to scale worse than $O(n_{\textrm{pts}}^2)$. This can be 
attributed to several factors. For the iterative solvers, increasing $N$ enlarges the domain diameter 
measured in wavelengths, $N_{\lambda}$; one therefore expects the GMRES iteration count to grow with
$N_{\lambda}$~\cite{galkowski2019wavenumber}, and the FMM’s expansion order/tree depth also increase 
with frequency, inflating the per-iteration prefactor. For the fast direct solver, the cost of obtaining 
a compressed representation of the inverse is highly dependent on the size of the linear system to be 
inverted directly after all near interactions have been compressed. Similar to the results 
in~\cite{sushnikova2023fmm}, the scaling of $O(n_{\textrm{pts}}^2)$ may in part be explained by the 
increased size of this linear system with increasing $N_{\lambda}$.

\subsection{A large-scale example}
As a final example, we compute the scattering matrix~$\mathcal{S}$ for the geometry in \figref{fig:quilt_soln}. 
The domain is constructed in the same manner as the geometries from the previous section but uses a perturbed 
lattice of $20 \times 20$ points with a spacing of $24$ in both directions, and perturbations drawn uniformly 
from $[-7.8, 7.8] \times [-7.8, 7.8]$. The final geometry contains~$N=380$ components with 12 external ports, 
and~$\sum_i\sum_p \widetilde M_{i,p} = 10890$ of modes. It contains 82 corner junctions, 164 triple junctions, 
and 134 quadruple junctions. The average number of points on each component $\tilde\Omega_{i}$ was $3498$, 
the average number of modes $\widetilde{M}=29$, and the maximum number of modes in any 
channel $\max_{i,p} \widetilde M_{i,p}$ was $20$.

The scattering matrix for this circuit with Neumann boundary conditions was computed using 
Algorithm~\ref{alg:1} with the computation of~$\Strunc$ requiring 584 seconds, of which 0.23 seconds 
were spent solving the sparse linear system. The cost of computing~$\cS$ using~\eqref{eq:scattering_mat} 
for this circuit was negligible. In~\figref {fig:quilt_soln}, we plot the solution to~\eqref{eq:piece_helm} 
and \eqref{eq:piece_jump} with incoming data defined by a random collection of coefficients $\vec{c}_{-}$. 
The solution $u$ in $\Omega_{0}$ is computed using the impedance data on the interior edges which is 
readily available through the solution of ~\eqref{eq:final_i2i_glue}. Once the interior impedance data 
is known, we use the single layer representation to compute the solution in each component. For this example, 
solving the local systems and evaluating the field at 918213 grid points took 95.4 seconds. 
\begin{figure}
    \centering
    \includegraphics[width=0.9\linewidth]{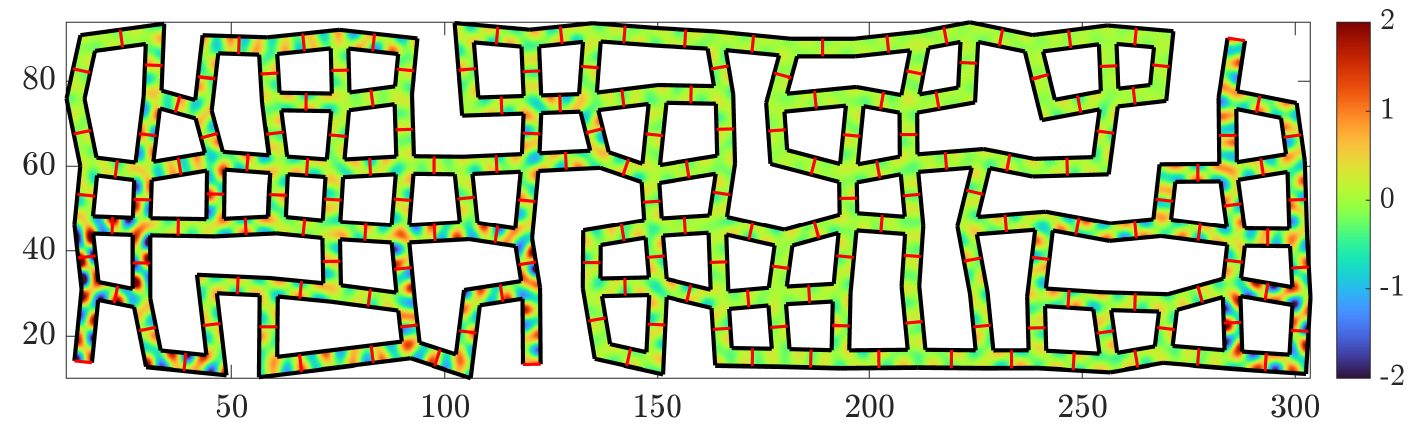}
    \caption{This figure shows the geometry used in our speed tests. In each test, we build 
    a $6\times n_x$ grid of simple components and study how the computational time scales with~$n_x$. 
    The red lines indicate the boundaries between adjacent components.}
    \label{fig:merge_quilt}
\end{figure}

\begin{figure}
    \centering
    \includegraphics[width=0.9\linewidth]{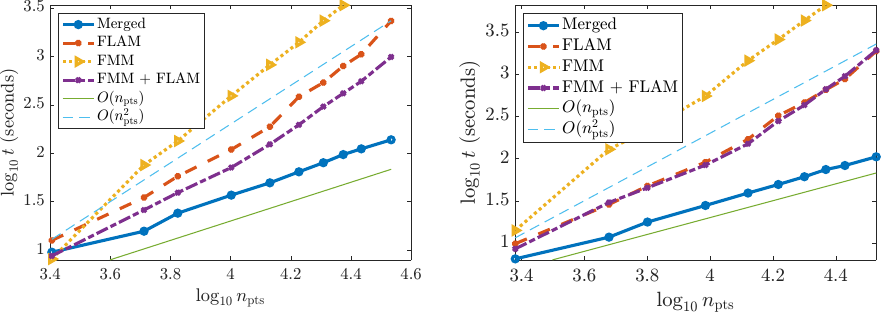}
    \caption{These figures show the time to find $\Strunc$ for the geometry in \figref{fig:merge_quilt} 
    versus number of points in the global discretization as we increase $n_x$. The left and rights 
    figure shows the timings for the Dirichlet and Neumann boundary conditions respectively.}
    \label{fig:merge_timing}
\end{figure}
\begin{figure}
    \centering
    \includegraphics[width=0.65\linewidth]{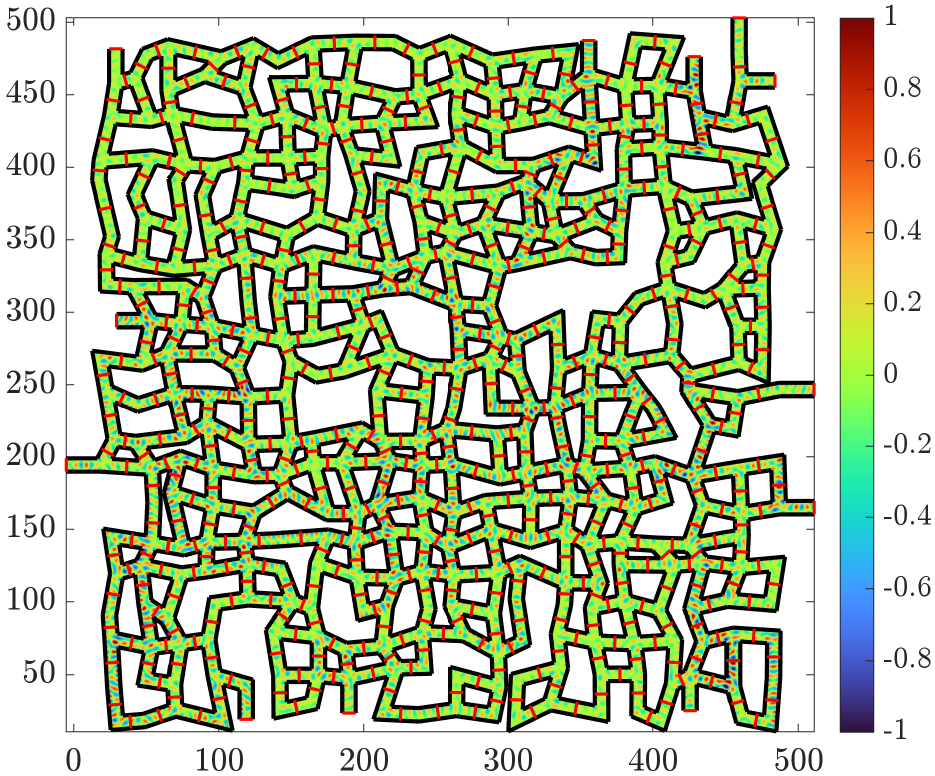}
    \caption{An example solution of the Neumann waveguide problem in our larger computational domain 
    formed from a~$20\times 20$ grid of simple components.}
    \label{fig:quilt_soln}
\end{figure}

\section{Conclusion}\label{sec:concl}
In this work, we present an analysis of and a new numerical method
for time-harmonic wave scattering problems in large metallic waveguide
systems in two dimensions. We first show that even though these waveguide
systems can have trapped modes, the projection of the solution onto
the propagating parts in the ports (which are infinite in extent)
is uniquely determined. The proposed numerical method, based on a
divide-and-conquer approach, first constructs the solution operators
of smaller subdomains in a compressed basis, then constructs the solution
in the whole device by enforcing continuity conditions across interfaces.
This approach outperforms FMM-accelerated iterative solvers, recursive
skeletonization based fast direct solvers, and hybrid preconditioned
iterative solvers by almost two orders of magnitude for devices 80
wavelengths in size in both directions.

For all the examples considered in this work, the size of the linear
system to be solved post compression was sufficiently small so that
dense inversion was the most computationally efficient. However, as
the complexity of the devices grows, fast methods would be required
for its inversion. The linear system to be inverted is block sparse
with non-zero entries on subdomains that share an interface --- enabling
the use of fast sparse solvers like UMFPACK~\cite{davis2004algorithm}.
Alternatively, the hierarchical part of surface HPS solvers which
relies on the local connectivity structure for merging solution operators
on unstructured meshes (as opposed to standard HPS solvers, for which
the data structure is a logical quadtree) can be easily adopted as
well.

The other natural extensions of this work include its extension to
the solution of metallic waveguide systems in three dimensions, and
to dielectric waveguide systems in two or three dimensions. While
the approach presented in this work extends easily to three dimensional
metallic waveguide systems, the extension to dielectric waveguide
systems is significantly more challenging owing to slow decay of the
lossy radiation into the exterior of the device. The identification
of an appropriate basis for the solution operators in smaller subdomains
which would result in its efficient compression remains an open area
of research.

\section*{Acknowledgments}
We would like to thank Leslie Greengard at Flatiron Institute and New York University 
and Jeremy Hoskins at University of Chicago for many useful discussions.  The Flatiron Institute 
is a division of the Simons Foundation.

\bibliographystyle{siamplain}
\bibliography{references.bib}

\end{document}